\documentclass[a4paper,10pt]{amsart}
\usepackage[utf8]{inputenc}
\usepackage{graphicx}
\usepackage[american]{babel}
\usepackage{pgf,tikz}
\usepackage[paperwidth=6.69in, paperheight=9.61in, bottom=1cm]{geometry}
\usepackage{amsfonts}
\usepackage{amsmath}
\usepackage{amssymb}
\usepackage{latexsym}
\usepackage{natbib}
\usepackage{amsthm}
\usepackage{indentfirst}
\usepackage{color}

\usepackage{ifpdf}

\usepackage{graphicx}

\theoremstyle{plain}
\newtheorem{theorem}{Theorem}[section]
\newtheorem{lemma}[theorem]{Lemma}

\newtheorem*{theorem*}{Theorem}
\newtheorem*{lemma*}{Lemma}
\newtheorem*{corollary*}{Corollary}
\newtheorem*{statement*}{Statement}
\newtheorem*{proposition*}{Proposition}
\newtheorem{remark*}[theorem]{Remark}

\theoremstyle{definition}
\newtheorem{definition}{Definition}    
     
\newtheorem{example}{Example}     

\newtheorem*{definition*}{Definition}
\newtheorem*{conjecture*}{Conjecture}
\newtheorem*{example*}{Example}

\theoremstyle{remark}
\newtheorem*{note*}{Note} 
\newtheorem*{case*}{Case}

\newcounter{fig}
\renewcommand{\figure}{\refstepcounter{fig}%
                  Fig. \arabic{fig}. }

\ifpdf
\DeclareGraphicsExtensions{.pdf, .jpg, .tif, .mps}
\else
\DeclareGraphicsExtensions{.mps, .jpg, .eps}
\fi

\begin{document}

\title{Grationality, with a Spoon}

\author{L. Jeneva Clark}

\address{Department of Mathematics, University of Tennessee,
Knoxville, Tennessee 37996-1320, {\tt E-mail: dr.jenevaclark@utk.edu}}

\begin{abstract} 
 The introduction of \textit{Grationality} at a 2025 sectional meeting of the Mathematical Association of America~\cite{MAAtalk} installed a handle on a concept akin to rationality of numbers, but in a geometric context. An \textit{nice} $n$-gon was defined to be a regular $n$-gon with side lengths that are natural numbers, and a number $n$ was defined to be \textit{grational} if and only if there exists a nice $n$-gon such that its area equals the sum of areas of $n$ congruent nice $n$-gons. This paper shows several examples of grational and non-grational numbers, followed by theorems about how the grationality of a number relates to its divisibility. Proofs of these theorems do not use high-powered tools, but rely on geometric constructions, proportional reasoning, tiling, dissection, the Carpets Theorem, and proof by descent. In keeping with this simplicity, a.k.a. ``doing math with a spoon," images are heavily leveraged. The benefits of choosing simplistic tools are discussed.\end{abstract}

\maketitle

\section{Introduction} A wise egg once said, {“When I use a word, it means just what I choose it to mean — neither more nor less,”}~\cite{carroll1875through}. Humpty Dumpty knew that  providing a \textit{tabula rasa} for new ideas can help open minds, impelling unbiased clarity in mathematical communication and logic. \textit{Grationality}, a new word introduced at a 2025 Sectional Meeting of the Mathematical Association of America~\cite{MAAtalk}, helps approximate something akin to rationality of numbers, but based in a geometric context. In this paper, I attempt to use the stealth of geometry to sneak up on number theory, such as investigating the grationality of perfect-square integers. In doing so, I also radically advocate for ``doing math with a spoon," the art of using the simplest mathematical tools possible. This radically minimalist approach affords more opportunities for tangential discoveries, allows for broader participation, connects current ideas to math history, and leans into simplicity as one of the pillars of proof elegance.

%We define a \textit{nice} $n$-gon to be a regular $n$-gon with side lengths that are natural numbers. Also, we define that a number $n$ is \textit{grational} if and only if there exists a nice $n$-gon such that its area equals the sum of areas of $n$ congruent nice $n$-gons. 

\section{Spoons and Carpets}

\subsection{Some Mathematical Background}
In the 1950s, Stanley Tennenbaum, a student of John H. Conway, executed a strinkingly elegant proof of the irrationality of the square root of two~\citep{conway2006power}, which will be described in Section~\ref{Tennenbaum}, and some have worked on generalizing its technique~\citep{Miller01042012}. Building on this work, I wanted to move the goalpost, tackling \textit{grationality} instead of rationality.  The word \textit{grational} is a loose blend of the words geometry and rational, hinting at a number-theoretical concept clothed in the intuition of geometry. By creating a few new words, I hoped to push forward mathematical thought with improved clarity, with a byproduct of empowering math students to similarly take ownership of mathematical ideas by coining new words. To this end, I define the following terms, within the context of Euclidean geometry: 
\begin{enumerate}
    \item A nice $n$-gon is a regular $n$-gon with side lengths that are natural numbers. 
    \item An integer $n$ is grational if and only if there exists a nice $n$-gon such that its area equals the sum of areas of $n$ congruent nice $n$-gons.
    \end{enumerate}

\subsection{A Connection to Mathematics History}

The accessibility and simplicity of the Tennenbaum proof offers broad inspiration to mathematics education, as it can be understood by very young mathematics learners. It hinges only on the well-ordering axiom of positive integers and on the Carpets Theorem, which will be stated as Theorem~\ref{thm:carpets}. Some suspect this simple strategy was connected to the early discovery of irrational numbers because of Plato's writings, which reveal that a young person discovered irrationality with rudimentary mathematical tools~\citep{conway2013extreme}. According to Plato, the requisite knowledge could be taught to young students in just one lesson and rationality/irrationality had been determined for square roots of integers from 3 to 17. 

Math history is a long mathematical conversation between humanity and itself. I believe that intentional engagement with past mathematics rewards us with the fitness required (both mathematical strength and flexibility) to enjoy healthy strides into future mathematics. For this reason, I dust off this question: Why might have the case-by-case journey of early mathematicians been limited to the interval $[3, 17]$? One popular conjecture bears the namesake of Theodorus, which is the pseudonym Plato gave the mathematician in his writings. Despite its popularity, the Spiral of Theodorus, as shown in Figure~\ref{fig:.4}, fails to explain why early mathematicians would have begun with 3; it begins with one. What if rationality were adjusted to a more geometric interpretation? After all, the smallest polygon has 3 sides. Might we envision a closer guess of what mathematical reasoning might have looked like a few thousand years ago?

\begin{center}\includegraphics[width=.6\textwidth]{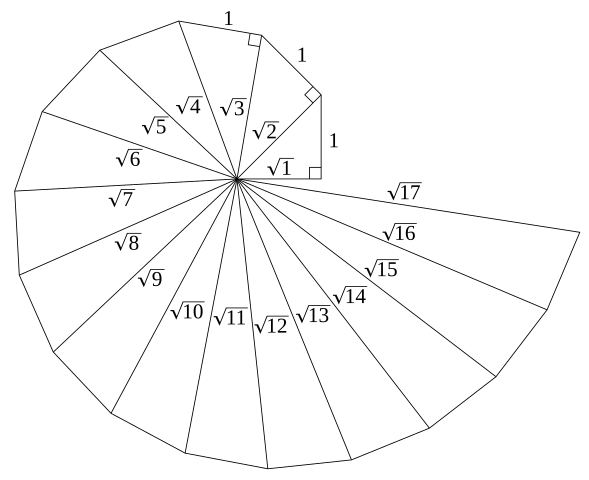}\\
      \figure The Spiral of Theodorus.
        \label{fig:.4}
  \end{center}

\subsection{Doing Math with a Spoon}
In this section, from the perspective of a mathematics educator, I strive to use simple tools of logic, rather than modern or high-powered techniques. This simplicity is like choosing to dig a swimming pool with a spoon, rather than with a shovel or a tractor. At first, this might sound foolish, but I am viewing mathematics through a lens of mathematics education. The heart of math itself beats only in the logic of the human mind. Thus, surprising simplicity both pleases aesthetic senses and ignites learning of mathematics, highly valued impacts in math education.

\begin{center}\includegraphics[width=\textwidth]{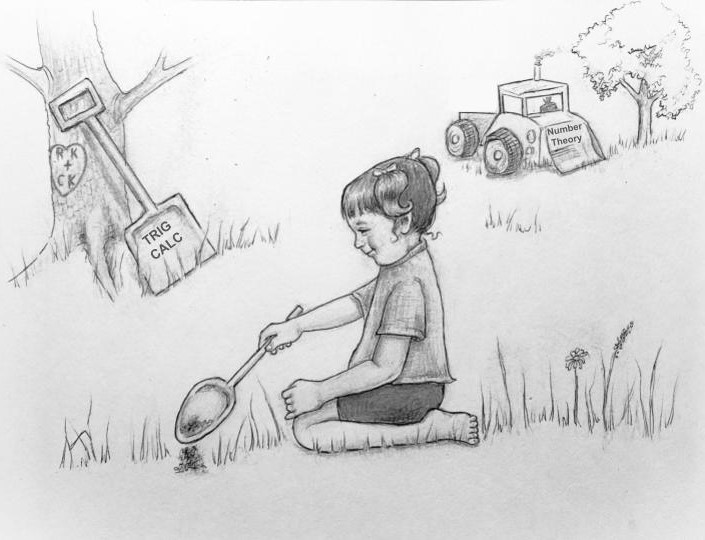}\\
      \figure A 3-year-old using a simple tool, drawing by Courtenay Kimberling
        \label{fig:.5}
  \end{center}

  Figure~\ref{fig:.5} depicts one of my earliest endeavors -- digging a swimming pool with a spoon at age 3. When I asked him for a swimming pool, my father, who often used tools like shovels and jackhammers around our estate, told me, ``Sure! You can dig a hole with anything!" He handed me his cereal spoon and instantly recruited his littlest construction worker. The smallness of a spoon granted dirt-moving power to even the smallest of hands. My control and autonomy over the earth afforded me more discovery. What cool stuff I found in the dirt! 
  
  For these reasons, autonomy, accessibility, and discovery, instead of doing mathematics that asks how big I can build a bulldozer, I favor the question, \textit{How small can I make my spoon?} In this paper, I stubbornly rely on spoon-level mathematical techniques -- geometric constructions, proportional reasoning, tiling, dissection, proof by descent, and only one additional theorem. The Carpets Theorem~\citep{andreescu2011mathematical} is low-hanging fruit for a young math learner. My 3-year-old self might have even been able to understand its statement and proof, as provided below, using both words and pictures to carry ideas.

\begin{theorem}[The Carpets Theorem]\label{thm:carpets}
Suppose that the floor of a room is completely covered by a collection of non-overlapping carpets such that the area of the room equals the combined areas of the carpets. If one were to move one or more of these carpets inside this room, then there may result region(s) where two carpets overlap each other and some region(s) that are not covered by any carpet. If so, the doubly carpeted area equals the uncarpeted area. 
  \end{theorem}

\begin{center}\includegraphics[width=.9\textwidth]{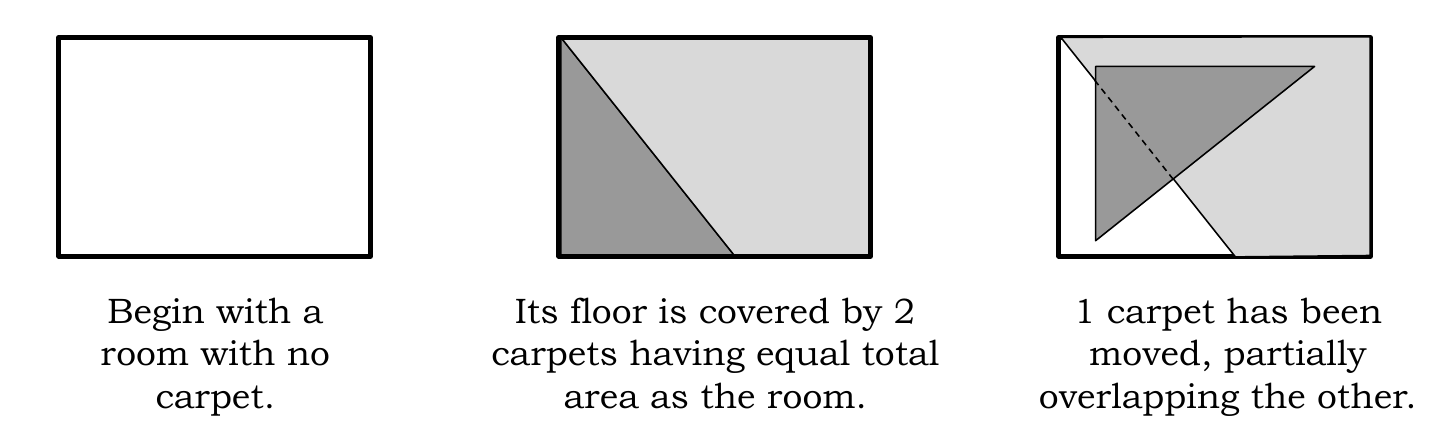}\\
      \figure An example of conditions where the Carpets Theorem holds.
              \label{fig:3}

  \end{center}

\begin{proof}
    Let a room be an arbitrary 2-dimensional shape with a non-empty interior. Let there be a collection of carpets inside the room such that the area sum of the carpets equals the area of the room. Suppose there is at least 1 region in which carpets overlap one another. Consider the different regions that result, as shown in Figure \ref{fig:5}. There are two different regions with the same triangular area, as shown in Figure \ref{fig:6}.   
    \begin{center}\includegraphics[width=.7\textwidth]{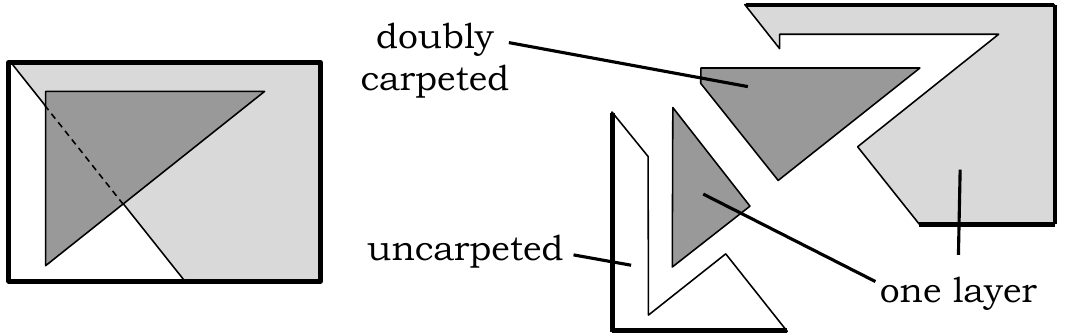}\\
      \figure Regions described by how many layers of carpet they have.
              \label{fig:5}
  \end{center}
  \begin{center}\includegraphics[width=.8\textwidth]{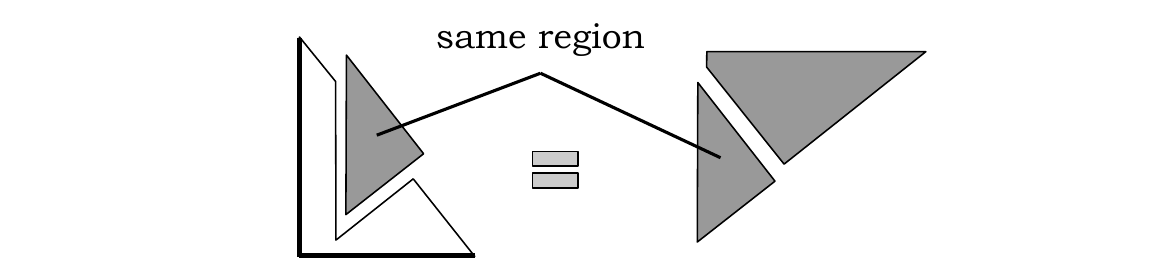}\\
      \figure Two regions are equal in area.
              \label{fig:6}
  \end{center}
 
In Figure \ref{fig:6}, each side of the ``equation" has a common term, implying the equality of the uncarpeted and doubly carpeted areas.
\begin{center}\includegraphics[width=.8\textwidth]{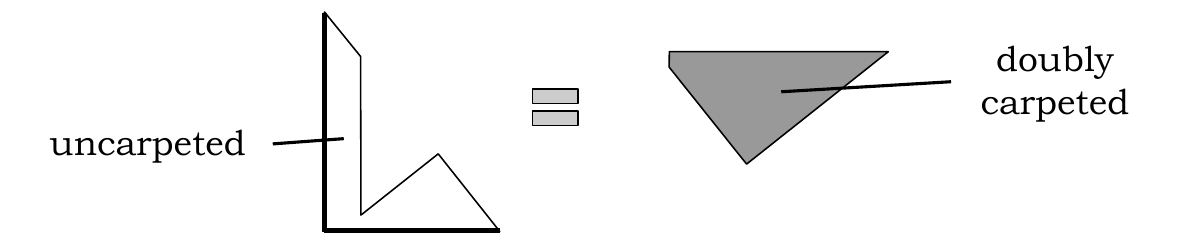}\\
      \figure The uncarpeted area equals the doubly carpeted area.
              \label{fig:7}
  \end{center}
  \end{proof}

 \subsection{Tennenbaum's Proof}\label{Tennenbaum}
  Tennenbaum began with the smallest case of an integer-by-integer square room and 2 integer-by-integer square carpets, such that the areas of the 2 carpets summed to room's area. He then put the 2 square carpets into opposite corners of the room, as in Figure~\ref{fig:7a}.

\begin{center}\includegraphics[width=.45\textwidth]{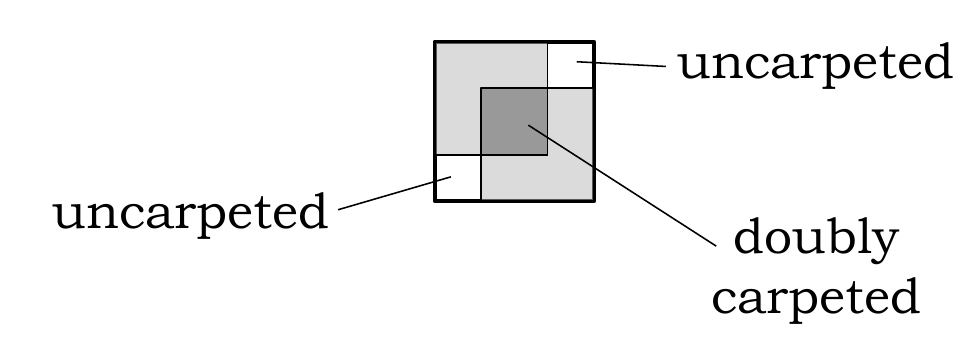}\\
      \figure Two square carpets in opposite corners of a square room.
              \label{fig:7a}
  \end{center}
By the Carpets Theorem, the uncarpeted area (2 squares) equals the doubly carpeted area (1 square). Because these squares have integer side lengths, this ``smallest" case is not the smallest, leading to a contradiction, implying $\sqrt{2}$ is irrational. 

Others have suggested generalizing Tennenbaum's proof to more polygons~\cite{Miller01042012, conway2006power}, but in this paper, I argue that viewing this as a base case for analogy is slightly amiss. Tennenbaum used 4-gon shapes with 2 carpets, but pairing 4 with 2 seemed a mismatch when the natural generalization path leads us to $n$-gon shapes with $n$ carpets, suggested by Conway~\citep{conway2006power} and further explored by Miller and Montague~\citep{Miller01042012}.  Thus, my stray from this proof is intentional, in an attempt to facilitate generalizability and geometric characterizations. I see the Tennenbaum proof as a springboard rather than a ladder, leaping to a more generalizeable pattern. 
\subsection{Niceness and Grationality, With a Spoon}
\begin{definition}
    \label{def: nicegon}
In Euclidean geometry, a \textit{nice-gon} is a regular polygon with integer side lengths. A \textit{nice n-gon} is a nice-gon with $n$ sides.

\end{definition}
\begin{definition}
    \label{def: grational}
An integer $n$ is \textit{grational} if and only if there exists a nice $n$-gon such that its area equals the sum of areas of $n$ congruent nice $n$-gons. In other words, suppose you have 2 nice $n$-gons for some integer $n\ge3$, and the areas of these nice-gons are $A$ and $B$; the grationality of $n$ means $A=nB$.
\end{definition}
\begin{example}
    Four is {grational}.\label{ex: 4}
\end{example}
\begin{proof}
A 2-by-2 square is a nice 4-gon with area 4, which equals the summed area of 4 1-by-1 squares, which are nice 4-gons. 

        \begin{center}\includegraphics[width=.7\textwidth]{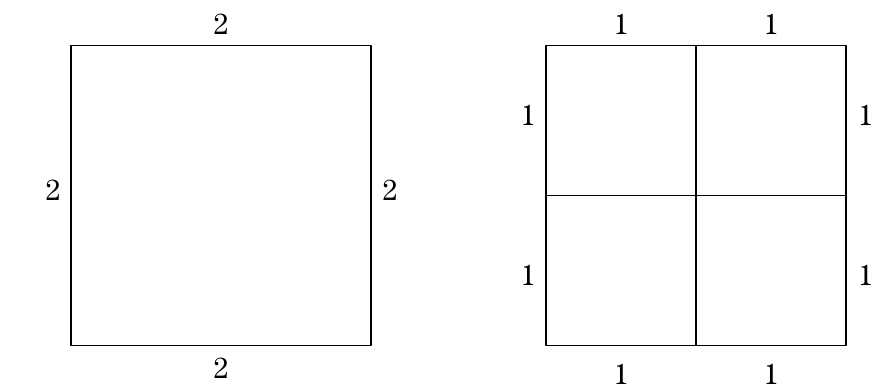}\\
      \figure A 2-by-2 square has the same area as 4 1-by-1 squares.
              \label{fig:2}
  \end{center}
\end{proof}

\noindent\textbf{Note.} Two is not grational. In Euclidean geometry, a 2-gon is not possible. Since Plato's narrative began with an $n=3$ case, this does not preclude grationality from being a candidate for the mysterious early work.
\begin{example}
    Three is not grational.\label{ex: 3}
\end{example}

\begin{proof}
    (Proved independently, but acknowledging similarity to a published proof of the irrationality of $\sqrt{3}$ found in~\cite{Miller01042012}.) Suppose not. Suppose three is grational. Then there would exist 3 nice 3-gons (equilateral triangles with integer side lengths) whose areas sum to the area of 1 nice 3-gon. Let us think of the triangle of area \textit{A} as the room and the triangles of area \textit{B} as the carpets, as shown in Figure \ref{fig:8}.

\begin{center}\includegraphics[width=.9\textwidth]{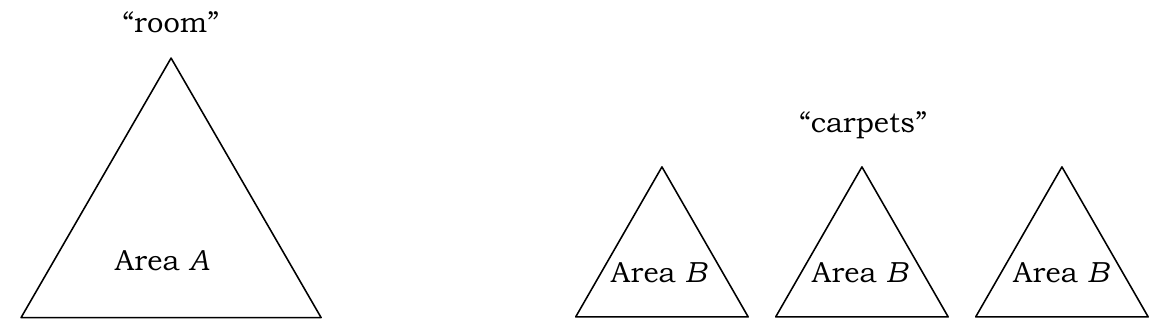}\\
      \figure One nice 3-gon has area $A=3B$, where $B$ is the area of another nice 3-gon. \label{fig:8}
  \end{center}
  
The side lengths are integers, and positive integers are bounded below. Thus, there must be a smallest possible case such that 1 nice 3-gon has triple the area of another. Suppose that $A$ is the smallest number such that one nice 3-gon has area $A=3B$, where $B$ is the area of another nice 3-gon. We shall call the integer side length of the larger triangle $a$ and the integer side length of the smaller triangles $b$.

    Overlay each of the 3 smaller triangular carpets onto the triangular room such that they share exactly one distinct vertex with the room and such that the interior of each carpet is contained in the interior of the room.
  \begin{center}\includegraphics[width=.3\textwidth]{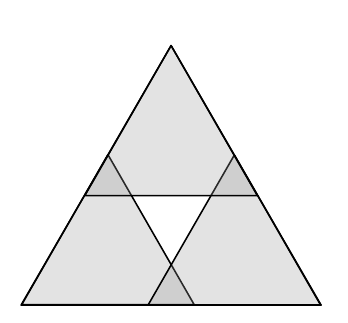}\\
      \figure The 3 smaller nice-gons (carpets) are overlaid onto the larger nice-gon (room).
              \label{fig:9}
  \end{center}
To prove that there would be 3 doubly carpeted areas, as shown in Figure~\ref{fig:9}, tiling of equilateral triangles can be used to show that $a<2b$, shown in Figure \ref{fig:10}.
  \begin{center}\includegraphics[width=.8\textwidth]{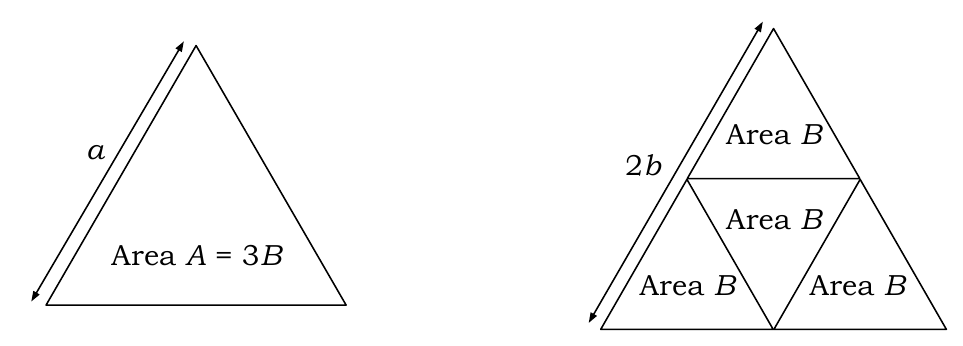}\\
      \figure A triangle of side length $2b$ has an area of $4B$, which is greater than $3B=A$.
              \label{fig:10}
  \end{center}
\noindent Because $A=3B<4B$, which is the area of a triangle with side length $2b$, we know $a<2b$. While this may seem obvious, I found value in verifying that these needed inequalities not only are true, but can be verified ``with a spoon," with the logic of a child. This rests on the intuitive principle that if two shapes are similar but have different sizes, then the one with the larger side length has the larger area. This fits in my spoon.

By applying the Carpets Theorem to Figure~\ref{fig:9}, the doubly carpeted areas (3 triangles) must equal the uncarpeted area (1 triangle). If these are nice 3-gons with the smaller ones being mutually congruent, then this would be a contradiction to our assumption that Figure \ref{fig:8} shows the smallest case.

Each doubly carpeted region is an equilateral triangle, as each shares two vertex angles with the carpets by which they were constructed. The doubly carpeted regions are congruent since each has a side length of $2b-a$, shown in Figure \ref{fig:11}. Because $a$ and $b$ are integers, $2b-a$ is also an integer, implying the doubly carpeted regions are nice-gons.
\begin{center}\includegraphics[width=.28\textwidth]{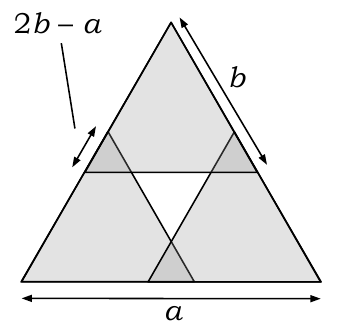}\\
      \figure The three doubly carpeted triangles have a side length of $2b-a$.
              \label{fig:11}
  \end{center}
\begin{center}\includegraphics[width=.4\textwidth]{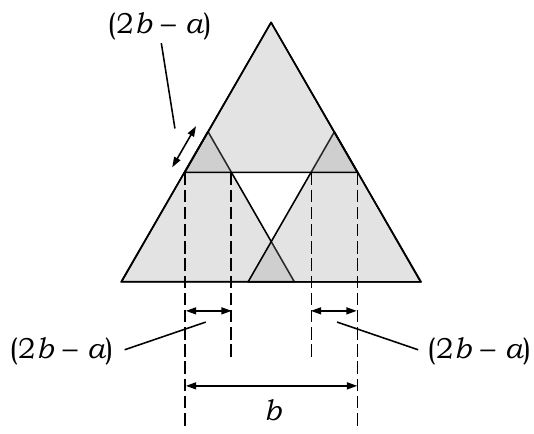}\\
      \figure The side-length of the uncarpeted triangle is $b-(2b-a)-(2b-a)$.
              \label{fig:12}
  \end{center}

The edge of the uncarpeted region, observable from Figure \ref{fig:12}, is $2a-3b$. A tiling technique similar to the one shown in Figure~\ref{fig:10} can show that an equilateral triangle of side length $2a$ has a larger area than an equilateral triangle of side length $3b$, which flows from the observation that $4A=4(3B)=12B>9B$. Because $2a > 3b$ and because $a$ and $b$ are integers, then $2a-3b$ is a positive integer, and the uncarpeted region is a nice 3-gon. We have shown a contradiction. Thus, three is not grational.
\end{proof}

\begin{example}
    Five is not grational.\label{ex: 5}
\end{example}
\begin{proof}
(Proved independently, but acknowledging similarity to a published proof of the irrationality of $\sqrt{5}$ found in~\cite{Miller01042012}.)
    Suppose not. Five is grational. There must exist five nice 5-gons (regular pentagons with integer side lengths) whose areas sum to the area of one nice 5-gon. Again, let \textit{A} be the area of the room and \textit{B} be the area of each carpet, each of which will be pentagonal with integer side lengths \textit{a} and \textit{b}, respectively, as shown in Figure \ref{fig:14}, and that this is the smallest case.
\begin{center}\includegraphics[width=.6\textwidth]{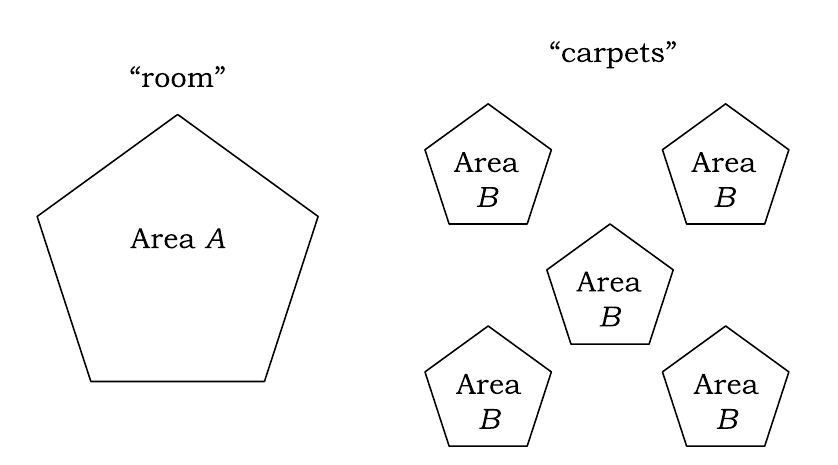}\\
      \figure A nice 5-gon has area $A=5B$, where $B$ is the area of another nice 5-gon.
              \label{fig:14}
  \end{center}

Overlay each of the pentagonal carpets onto the pentagonal room such that each shares exactly one distinct vertex with the room and such that the interior of each carpet is in the interior of the room, as shown in Figure \ref{fig:15}.
 \begin{center}\includegraphics[width=.3\textwidth]{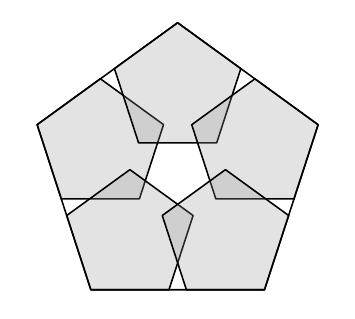}\\
      \figure The five smaller nice-gons (carpets) overlaid onto the larger nice-gon (room).
              \label{fig:15}
  \end{center}
To verify that there would indeed be doubly carpeted regions as they appear in Figure \ref{fig:15}, tiling arguments can be made, which compare areas of similar shapes. To provide one such example, how can we be certain that there are 5 triangular regions of uncarpeted floor? Consider Figure~\ref{fig:17}. The area of the left-hand pentagon, with side length $2b$, equals that of 20 tiled isosceles triangles, which is a smaller area than $A=5B$, which equals the area of 25 such isosceles triangles, shown on the right side of the image. This logic rests on the ideas that any regular $n$-gon can be partitioned into $n$ isosceles triangles, each with a central angle of $2\pi/n$, and that isosceles triangles can be tiled by smaller similar isosceles triangles, ideas which fit into my spoon.

\begin{center}\includegraphics[width=.7\textwidth]{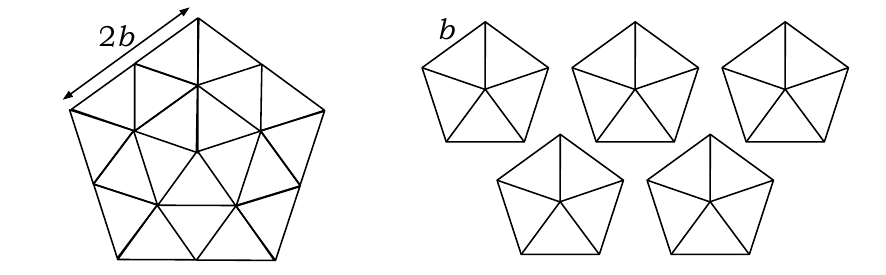}\\
      \figure If $a=2b$, then $A$ (20 triangles) would be less than $5B$ (25 triangles).
              \label{fig:17}
  \end{center}

The right side of Figure \ref{fig:18} shows the doubly carpeted and uncarpeted regions: 6 uncarpeted regions, 5 triangles and one pentagon, and 5 doubly carpeted quadrilaterals. Notice this did not result in $n$ doubly carpeted $n$-gonal regions and 1 $n$-gonal uncarpeted region, as resulted in the proof that 3 was not grational. 
\begin{center}\includegraphics[width=.6\textwidth]{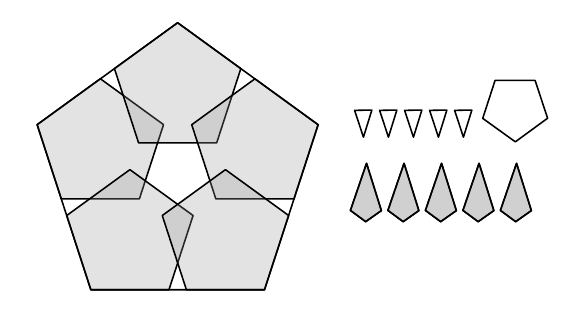}\\
      \figure Uncarpeted and doubly carpeted regions are shown on the right.
              \label{fig:18}
  \end{center}

By examining construction and using parallel lines, the small pentagon highlighted in Figure \ref{fig:20}(a) is regular. Symmetry and similar argumentation can show congruence of the 5 small pentagons shown in Figure \ref{fig:20}(b).

\begin{center}\includegraphics[width=.8\textwidth]{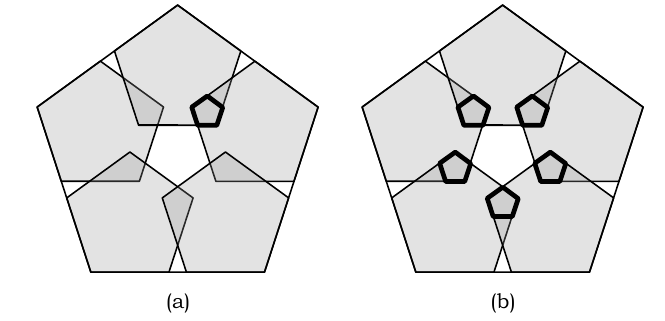}\\
      \figure (a) The bold pentagon is regular. (b) These five pentagons are congruent.
              \label{fig:20}
  \end{center}

If we were to cut and remove a piece of carpet from each doubly carpeted region, leaving those five congruent pentagons behind, would the cut carpet remnant perfectly fit in the uncarpeted triangular region, as shown in Figure~\ref{fig:21}? We need to determine if the newly created pentagons would have sides of length $a-2b$.
 \begin{center}\includegraphics[width=.7\textwidth]{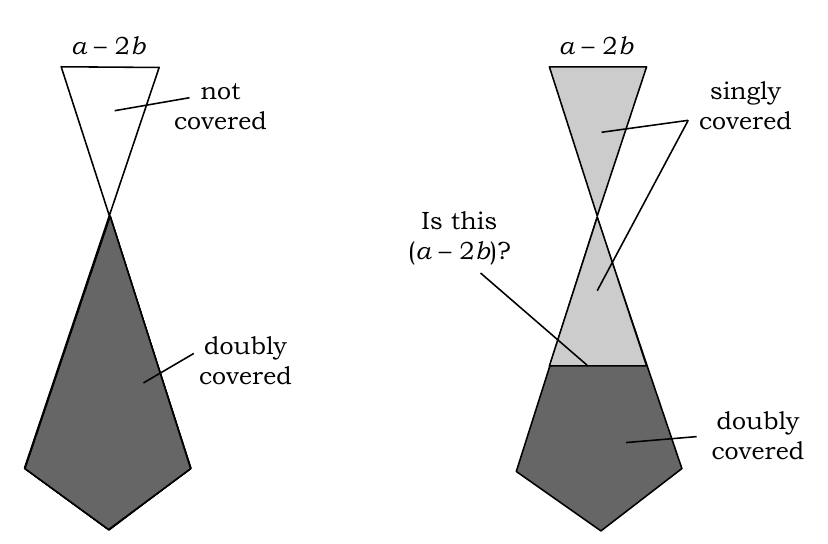}\\
      \figure 1 layer of each doubly carpeted region cut and moved.
              \label{fig:21}
  \end{center}

Let $d$ be the length of each carpet's diagonal, which connects any two non-adjacent vertices of the pentagon of side length $b$. Extend four sides of the room by a length of $b$ beyond the room's vertices, as shown in Figure \ref{fig:22}(a), and construction of parallel lines can lead us to an expression for $d$ in terms of $a$ and $b$, namely $(a+b)/2$.
 \begin{center}
      \includegraphics[width=.7\textwidth]{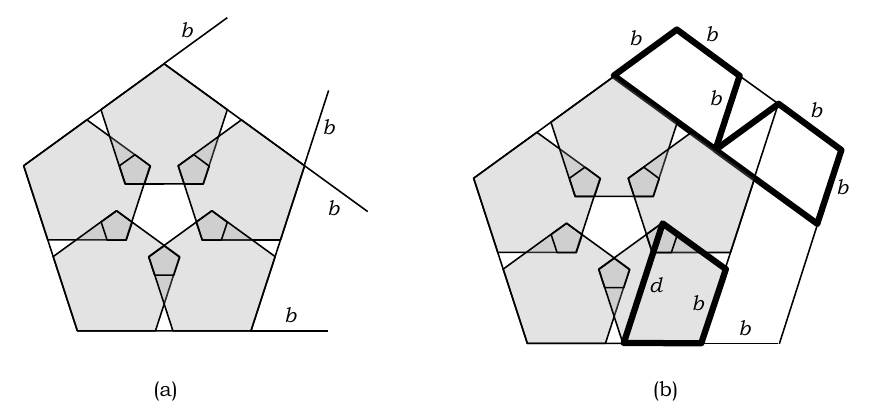}\\
      \figure (a) Extend sides of the room by a length of $b$. (b) Bold shapes are congruent. 
              \label{fig:22}
\end{center}

\noindent A side note: Arriving at the image in Figure~\ref{fig:22}b followed a bit of roleplay. While stubbornly resisting trigonometry or known facts about pentagons (shovels or wheelbarrows), I was able to fit this step into my spoon. To do this, I printed out the image of the pentagonal room and carpets, I picked up a ruler, and I pretended I was child who knew some basic facts about parallel lines. After drawing parallel lines all over the page, it was an orderly mess, and then I noticed the relationship needed.
 
\begin{center}\includegraphics[width=.3\textwidth]{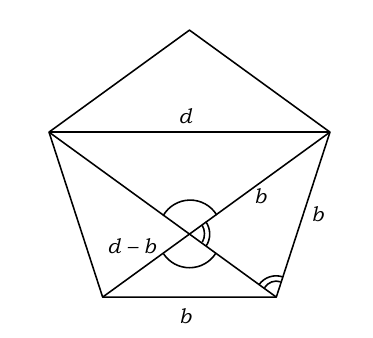}\\
      \figure The 3 drawn diagonals create similar isosceles triangles.
              \label{fig:24}
  \end{center}
  Similar isosceles triangles created by the pentagon's diagonals yield the proportion $d/b=b/(d-b)$. Rewriting and substitution yields $a^2=5b^2$. (Foreshadowing note: In Theorem~\ref{thm:quadratic}, I prove a more general version, $a^2=nb^2$, with a spoon.)

  It remains to determine if the small doubly carpeted pentagons have a side length equal to $a-2b$. The expression shown in Figure \ref{fig:25}(b) comes from subtracting expressions in Figure \ref{fig:25}(a). 

   \begin{center}
  \includegraphics[width=.8\textwidth]{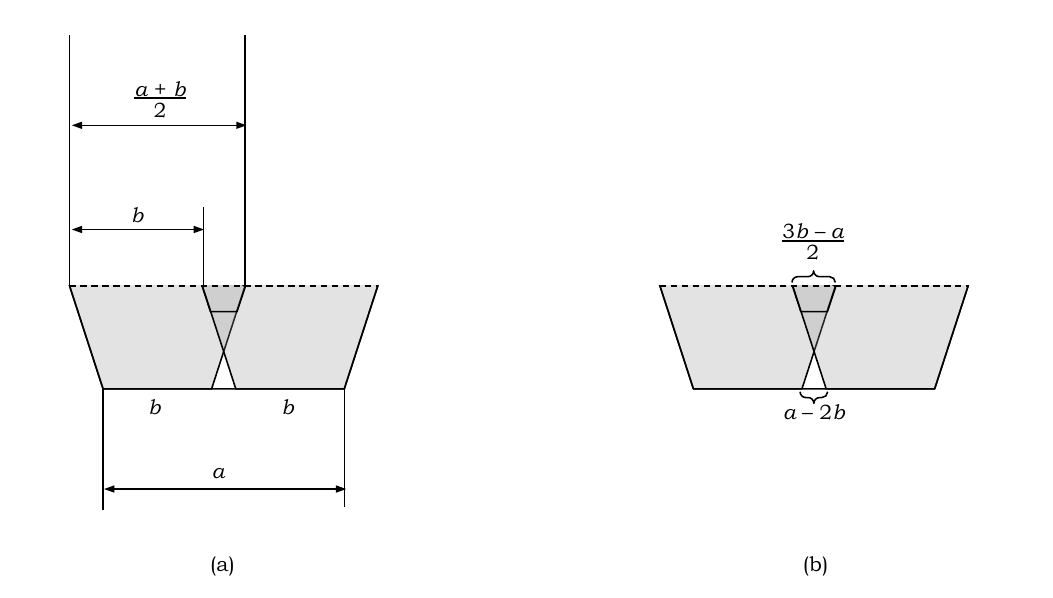}\\
      \figure (a) The bottom part of the room, truncated at the diagonals of 2 carpets. (b) The diagonal length of the small doubly carpeted pentagon is $(3b-a)/{2}$.
              \label{fig:25}

  \end{center}

Because the smallest pentagon is similar to the pentagon of side length $b$, a proportion gives us that the side length of the smallest pentagon shape is $(3b^2-ab)/(a+b)$, which can be rewritten as:
\[\frac{3b^2-ab}{a+b}\cdot \frac{a-b}{a-b}=\frac{b(4ab-3b^2-a^2)}{a^2-b^2}.\]

Because $a^2=5b^2$, we can substitute and arrive at the following.
\[x=\frac{b(4ab-8b^2)}{4b^2}=a-2b.\]

Therefore, the smallest pentagons indeed have a side length of $a-2b$, which will allow the dissection and repositioning of one layer from the doubly covered region, as shown in Figure \ref{fig:21}. Because we knew, by the Carpets Theorem, that the shaded regions in Figure \ref{fig:26}(a) have areas that sum to equal the sums of the unshaded regions, we now know that, in our new arrangement, that the doubly carpeted regions in Figure \ref{fig:26}(b) have an area sum equal to that of the uncarpeted pentagon.

\begin{center}\includegraphics[width=\textwidth]{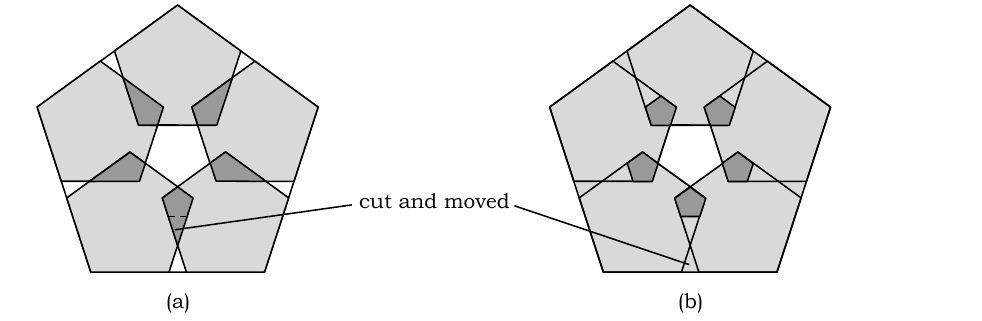}\\
      \figure If the doubly carpeted quadrilaterals in (a) are cut and rearranged to cover the uncarpeted triangles, then the room will look like (b).
              \label{fig:26}
  \end{center}

Because $a$ and $b$ are integers, $a-2b$ is an integer. By construction, the side length of the uncarpeted pentagon is $b-(a-2b)-(a-2b)$, which is $5b-2a$, also an integer. A tiling argument that compares the areas of similar isosceles triangles of base lengths $2a$ and $5b$ can verify that $5b-2a$ is positive, making the uncarpeted area a nice-gon. This is a contradiction to the assumption that we began with the smallest case. Thus, five is not grational.
\end{proof} 

The proof that 6 is not grational increases in complexity in a few ways. For brevity, I highlight the ways in which the proof is notably different from the preceding one. Overlaying six hexagonal carpets of equal summed area to that of their hexagonal room results in Figure \ref{fig:34}, with 7 uncarpeted regions and 6 doubly carpeted regions.

  \begin{center}
\includegraphics[width=.45\textwidth]{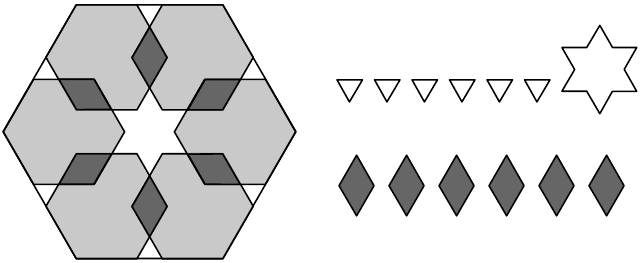}\\
      \figure 7 uncarpeted regions and 6 doubly carpeted regions result.
              \label{fig:34}
  \end{center}
The symmetries of hexagons and a simple dissection allow us to cut part of each carpet and move it to fit each uncarpeted triangle.
     \begin{center}
\includegraphics[width=.45\textwidth]{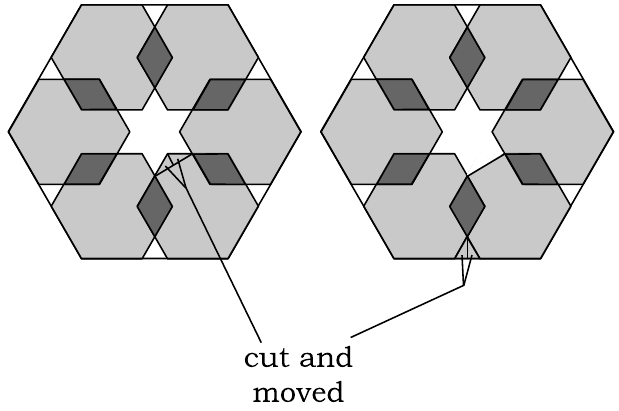}\\
      \figure A triangle of carpet moved to a uncarpeted region with the same area.
              \label{fig:38}
\end{center}
The proof that 6 is not grational required a dissection that showed each doubly carpeted quadrilateral rearranges into a regular hexagon, shown in Figure~\ref{fig:40}.

         \begin{center}
\includegraphics[width=.6\textwidth]{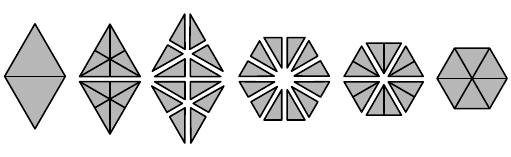}\\
      \figure This rhombus was dissected and rearranged to form a hexagon.
              \label{fig:40}
  \end{center}
With this dissection, the newly formed hexagon's side length, after a few similar-triangle observations, was found to be $(d/(3b))(3b-a)$. The proof then required showing $d$ to be divisible by $3b$ via hexagonal tiling and similar triangles.
  \begin{center}
\includegraphics[width=.8\textwidth]{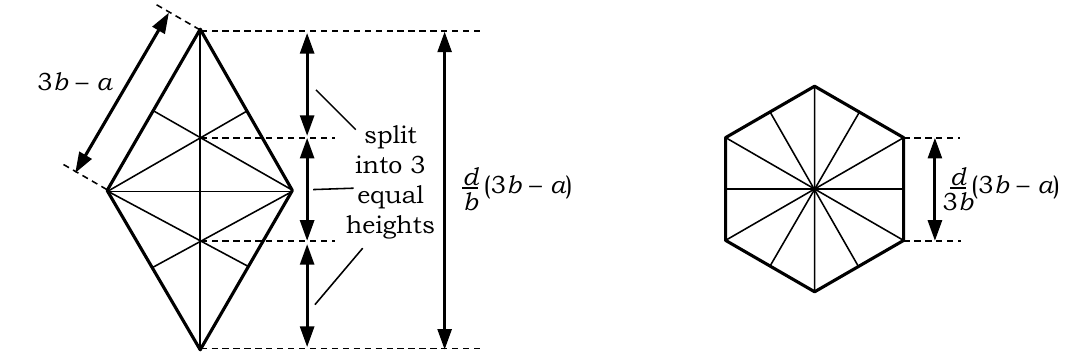}\\
      \figure The rhombus dissection splits its height into three equal sections.
              \label{fig:49}
  \end{center}
  Spoonful-by-spoonful, the proof led to a contradiction; 6 doubly carpeted nice 6-gons had an area sum equal to that of 1 uncarpeted nice 6-gon, as in Figure~\ref{fig:41}, which was smaller than the assumed smallest case.
      \begin{center}
\includegraphics[width=.5\textwidth]{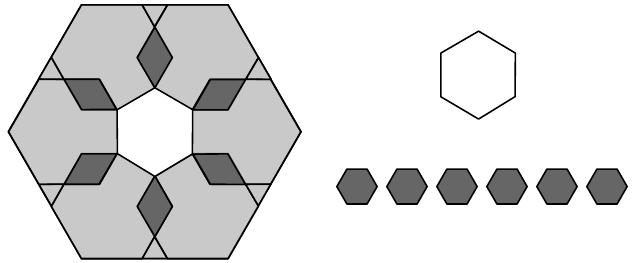}\\
      \figure One uncarpeted region and 6 doubly carpeted regions result.
              \label{fig:41}
  \end{center}
  In summary, the proof that six is not grational required a more complex dissection, which required additional similarity-derived proportions. It is noteworthy that the tile-ability of regular hexagons was leveraged more than once, a strategy not possible with a 7-gon. 
  
  Figure~\ref{fig:41.8} shows the cases for $n=7$ and $n=8$, appearing to increase in complexity and to require interesting dissections. In particular, the doubly carpeted areas in the $n=7$ case appear to be less symmetric. Thus, the need to conjecture a generalization is well-motivated. Contrastingly, nine can be shown to be grational, as shown in Example~\ref{ex: 9}.

    \begin{center}
\includegraphics[width=.5\textwidth]{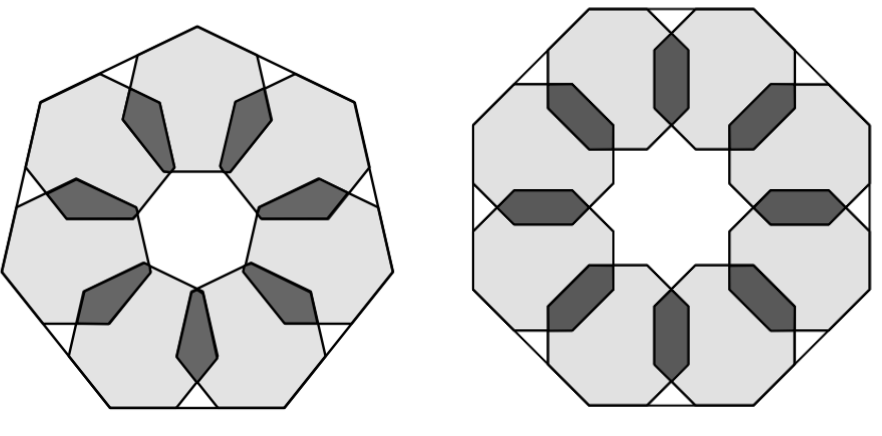}\\
      \figure The overlapping carpets for $n=7$ and $n=8$
              \label{fig:41.8}
  \end{center}
  
\begin{example}Nine is grational.\label{ex: 9}
\end{example}
\begin{proof}
Construct an isosceles triangle with angle measures $2\pi/9$, $7\pi/18$, and $7\pi/18$ with a shortest side length of one unit. Nine of these triangles can be arranged to form a regular nonagon of side length 1, which is a nice 9-gon. By triangular tiling, 81 of these triangles can be arranged to form a larger regular nonagon with side length 3, as shown on the left side of Figure \ref{fig:52}. This larger nice 9-gon has an area equal to the area sum of the 9 nice-gons shown on the right side of Figure~\ref{fig:52}. Thus, nine is grational.

    \begin{center}
\includegraphics[width=.8\textwidth]{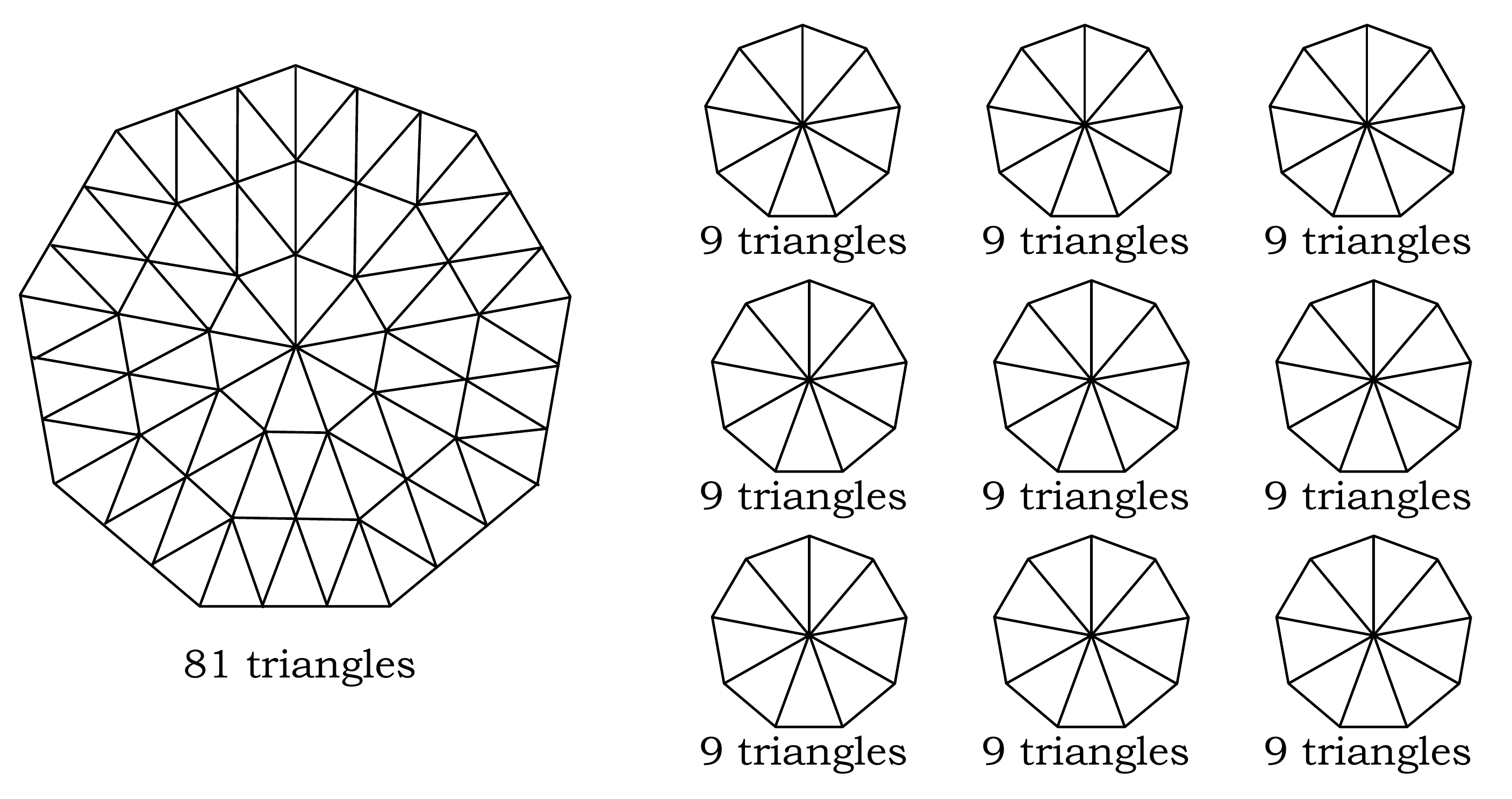}\\
      \figure The area of the nice-gon on the left equals the area sum of the 9 on the right.
              \label{fig:52}
  \end{center}
  
\end{proof}

The examples shown so far have set the scene for a generalization that most readers have already anticipated by noticing that perfect squares seem to be the grational ones. Even though it is well worthwhile to ask, \textit{Are all perfect squares grational?}, and, \textit{Are all grational numbers perfect squares?}, I find it more mathematically interesting to ask, \textit{Can I show both of these with only a spoon?}. I attempt to do this in Section~\ref{sub:3.7}

%The tiling used in the grationality of nine, in Example~\ref{ex: 9}, is an appropriate appetizer for Lemma~\ref{lem:J1}. Noticing that perfect square numbers seem to be grational and the intuition embedded in tileability led me to proof strategies for Theorems~\ref{thm:squaresgrat},~\ref{thm:quadratic}, and~\ref{thm:hard}, proven with a spoon in Section~\ref{sub:3.7}.

\subsection{Generalizing, with a Spoon}\label{sub:3.7}

In this section, I offer one lemma and three theorems that effectively show integers to be grational if and only if they are perfect squares. While this was neither difficult to notice nor prove, the challenge here was to prove it using only rudimentary logic and tools, in a Tennenbaum-esque manner. In keeping with my stubborn streak for simplicity, these are somewhat visual proofs, as convincing images I consider to be spoon-level mathematics. 

\begin{lemma}\label{lem:J1} If $k$ is an integer, then $k^2$ congruent isosceles triangles can be arranged to make one larger isosceles triangle that has (a) the same angle measures and (b) a base that is $k$ times the base of the smaller triangles.
\end{lemma}

\begin{proof} 
If $k$ is an integer, then $k^2$ items can be arranged into $k$ rows of $k$. Let those items be congruent isosceles triangles, as shown in Figure~\ref{fig:57.71}.
    \begin{center}
\includegraphics[width=.3\textwidth]{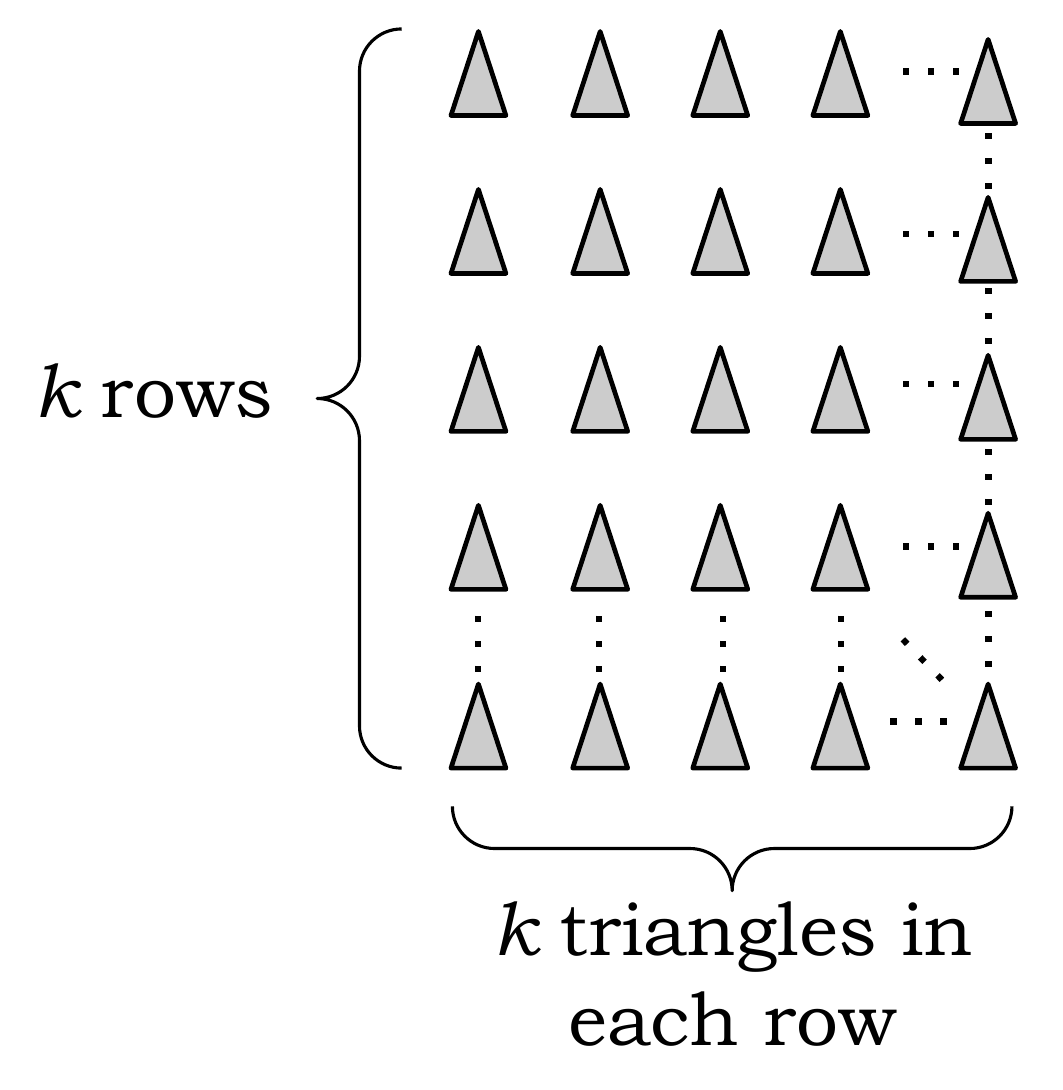}\\
      \figure A square of isosceles triangles arranged in $k$ rows of $k$.      \label{fig:57.71}
      \end{center}
This square of triangles can then be rearranged to form a triangle of tiled triangles, as shown in Figure~\ref{fig:57.8}. Each row after row 1 is formed from an L-shape of triangles that was adjacent to those used in the previous row. By this construction, row $n$ has $2n-1$ triangles in it for any integer $n$ in $[1,k]$.
         \begin{center}
\includegraphics[width=.8\textwidth]{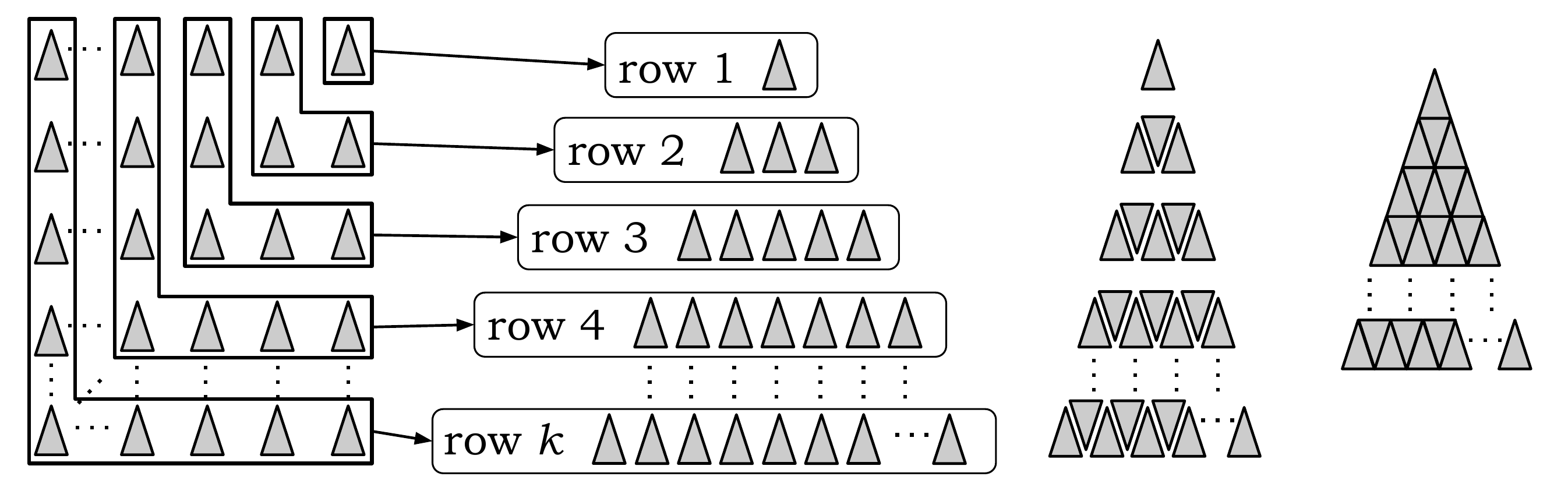}\\
      \figure A square of triangles can be rearranged to form a triangle of tiled triangles.        \label{fig:57.8}
      \end{center}

In each row, every second triangle is reversed in orientation, so that the triangles tile. The triangle resulting from gluing together all $k^2$ triangles is isosceles itself. In row $k$, $k$ of the small triangles are oriented with their bases along the bottom edge, which forms the base length of $k$ times the base of the smaller triangle.
\end{proof}
\begin{theorem}\label{thm:squaresgrat}
    If $n\ge2$ is an integer, then $n^2$ is grational.
\end{theorem}

\begin{proof}
Let $n\ge2$ be an integer. Construct $n^2$ isosceles triangles so that each one has both a central angle measuring $2\pi/(n^2)$ and a base length of 1. Arrange them adjacently with each $2\pi/(n^2)$-angle vertex concurrent, which becomes the center of a newly formed polygon, as shown in Figure~\ref{fig:53}. Gluing these together creates a nice $n^2$-gon with side length 1. Repeat this $n^2$ times to create $n^2$ carpets. These carpets have an area sum equal to that of $n^4$ constructed congruent isosceles triangles.

      \begin{center}
\includegraphics[width=.65\textwidth]{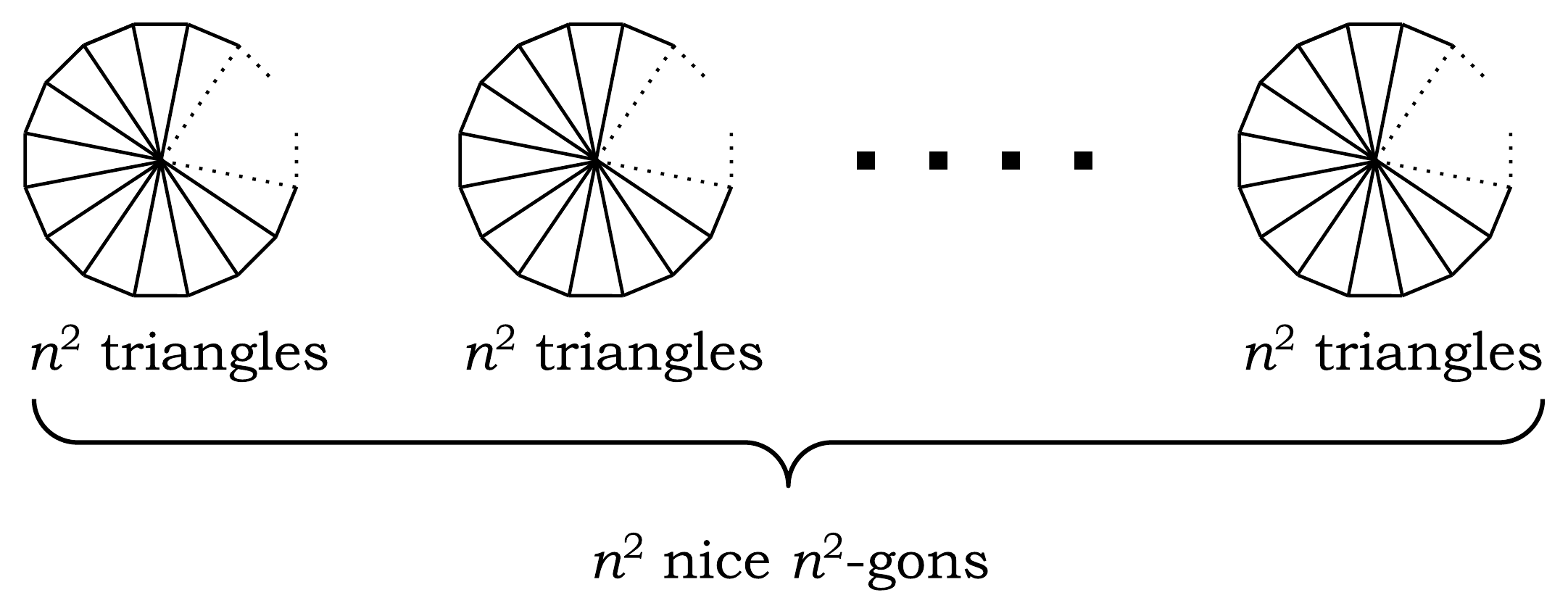}\\
      \figure An $n^2$ number of nice $n^2$-gons has an area sum of $n^4$.
              \label{fig:53}
  \end{center}  

To prove $n^2$ is grational, we need to construct a room that is a nice $n^2$-gon that also has a total area equal to the area sum of $n^4$ constructed congruent isosceles triangles of side length 1. First suppose we have triangles with both a central angle measuring $2\pi/(n^2)$ and a base length of 1. By Lemma~\ref{lem:J1}, $n^2$ of these triangles can be arranged to form a larger similar triangle whose base length is $n$. Repeat this arrangement $n^2$ times, forming $n^2$ isosceles triangles which can then be arranged into a nice $n^2$-gon with side length $n$, as shown in Figure \ref{fig:57}.

        \begin{center}
\includegraphics[width=.6\textwidth]{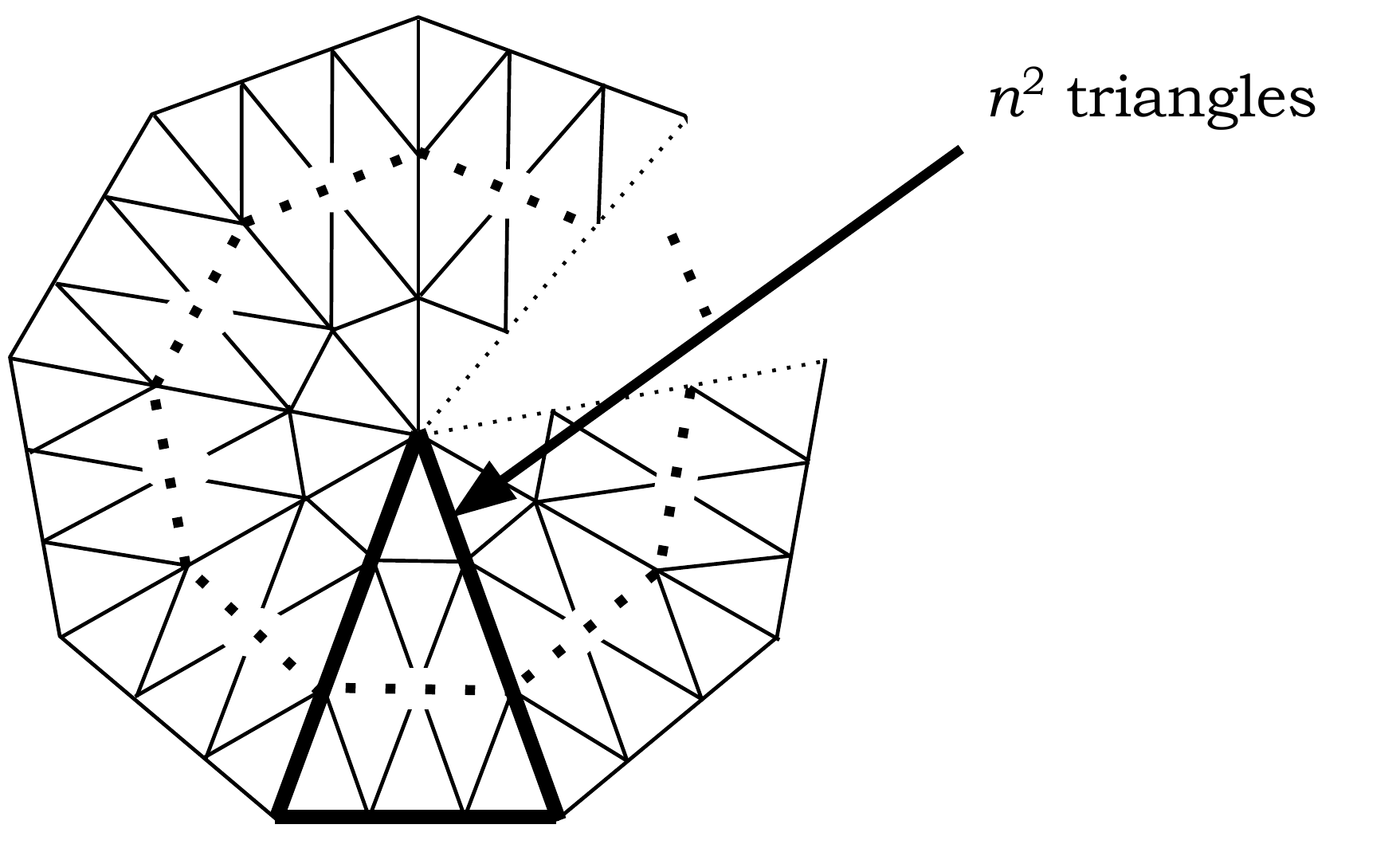}\\
      \figure A nice $n^2$-gon formed from $n^2$ triangles, each composed of $n^2$ triangles.         \label{fig:57}
      \end{center}

\end{proof}

\begin{theorem}\label{thm:quadratic} Suppose $n$ is grational. Grationality implies there exist 2 nice $n$-gons such that the area of one is $n$ times the area of the other. Let the larger nice-gon have side length $a$ and the smaller have side length $b$. Then, $nb^2=a^2$.
\end{theorem}

\begin{proof}
Suppose $n$ is grational. Increasing the side length of a regular polygon increases its area, and positive integers are bounded below, so if $n$-gons have positive integer side lengths, there must be a smallest possible number such that one nice $n$-gon has $n$ times the area of another. Suppose that $A$ is the smallest number such that one nice $n$-gon has area $A=nB$, where $B$ is the area of another nice $n$-gon. Let the side length of the larger $n$-gon be $a$ and the side length of the smaller $n$-gon be $b$. Because these are nice-gons, $a$ and $b$ are positive integers.

Every regular polygon can be dissected into isosceles triangles by connecting the polygon's center to each vertex. Thus, the areas $A$ and $B$, where $h_a$ and $h_b$ are the heights of such sub-triangles, are as follows:
\[A=n\cdot \frac{1}{2}\cdot a\cdot h_a\text{ and }B=n\cdot \frac{1}{2}\cdot b\cdot h_b\]

By the assumption of grationality, $A=nB$, which gives
\[n\cdot \frac{1}{2}\cdot a\cdot h_a=n^2\cdot \frac{1}{2}\cdot b\cdot h_b.\] By dividing by non-zero quantities, we get the ratio $h_a/h_b$ as equivalent to $nb/a$. However, we can arrive at a ratio for $h_a/h_b$ in another way. All nice $n$-gons are similar, and so must be the sub-triangles that result from dissections connecting the center to each vertex. This similarity tells us that the ratio $h_a/h_b$ is also ${a}/{b}$.
Equating these two different expressions for $h_a/h_b$ yields $nb^2=a^2$, or equivalently, $a^2/b^2=n$.

\end{proof}

\begin{theorem}\label{thm:hard} Every grational number is a perfect square. In other words, if $n$ is grational, then $n=k^2$ for some integer $k$.
\end{theorem} 

\begin{proof}
Suppose $n\ge3$ is grational. Define a \textit{unit triangle} to have base length 1, height $h$, and central angle $2\pi/n$. By Lemma~\ref{lem:J1}, $a^2$ unit triangles can be arranged to make one larger similar triangle with base length $a$ and with $a$ rows. Similarly, $b^2$ unit triangles can unite to form a similar triangle with base length $b$ and with $b$ rows. 
Let us consider the tiling of $a^2$ unit triangles to be a room and each tiled union of $b^2$ unit triangles to be a carpet, as shown in Figure~\ref{fig:57.92}. By Theorem~\ref{thm:quadratic}, $n$ of these carpets have the same area as the room.

        \begin{center}
\includegraphics[width=.65\textwidth]{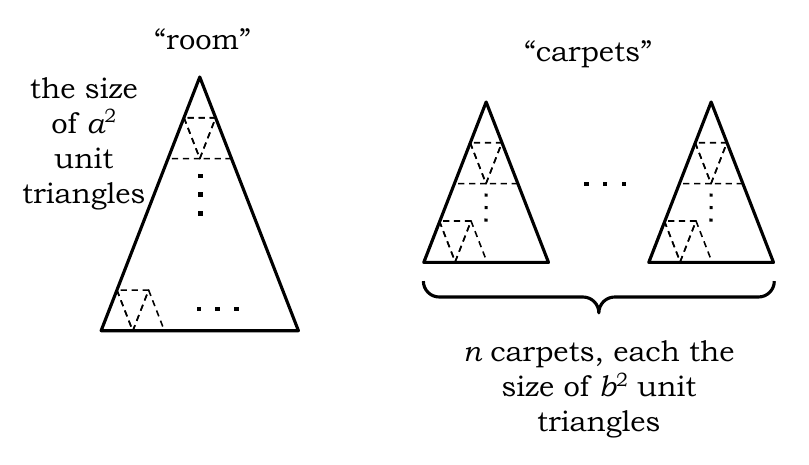}\\
      \figure The room and each carpet are the size of $a^2$ and $b^2$ unit triangles, respectively.     \label{fig:57.92}
      \end{center}

            It is possible that $n$ carpets will perfectly tile the room without requiring us to cut and rearrange any carpets. In order to affirm this, let us suppose not. Place the carpets in the room, beginning with one at the central-angle vertex. Below that, place remaining carpets, row by row in tiling formation, until there is insufficient vertical space to insert a next row. Each carpet is the union of $b$ rows of unit triangles, so each carpet would require at least $bh$ in height, regardless of whether the carpets were placed in the ``up" or ``down" orientation. In seeking a contradiction, I am assuming the carpets will not tile perfectly, row by row, in the room without dissection. Thus, the remaining uncarpeted strip would be less than $bh$ in height, as shown in Figure~\ref{fig:57.94}. 

  \begin{center}
\includegraphics[width=.53\textwidth]{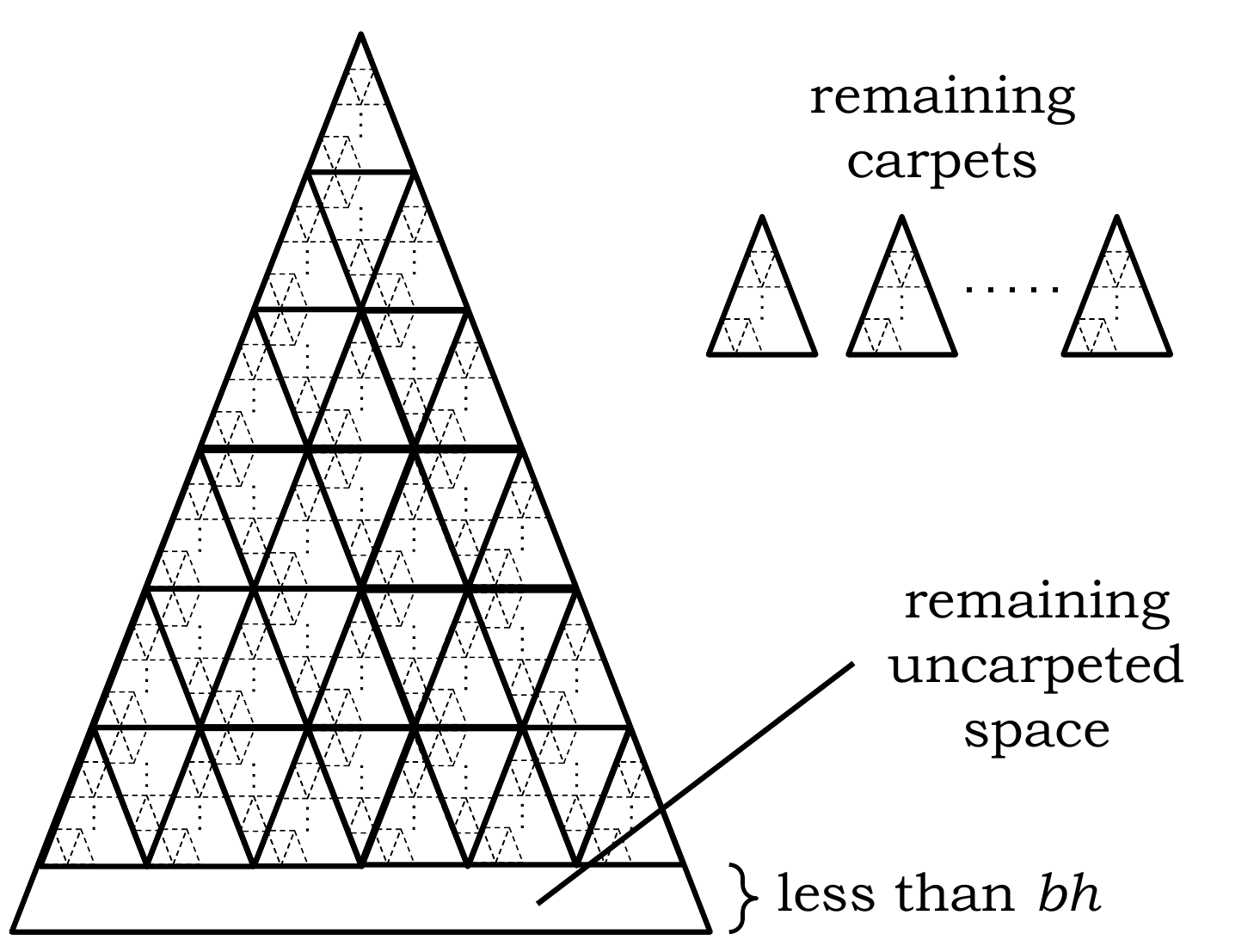}\\
      \figure The available space is less than $bh$ in height. \label{fig:57.94}
      \end{center}

      Next, decompose each of the remaining carpets into unit triangles, and insert these unit triangles into the uncarpeted region, one by one, left to right, row by row, top to bottom, as shown in Figure~\ref{fig:57.99}.

        \begin{center}
\includegraphics[width=.53\textwidth]{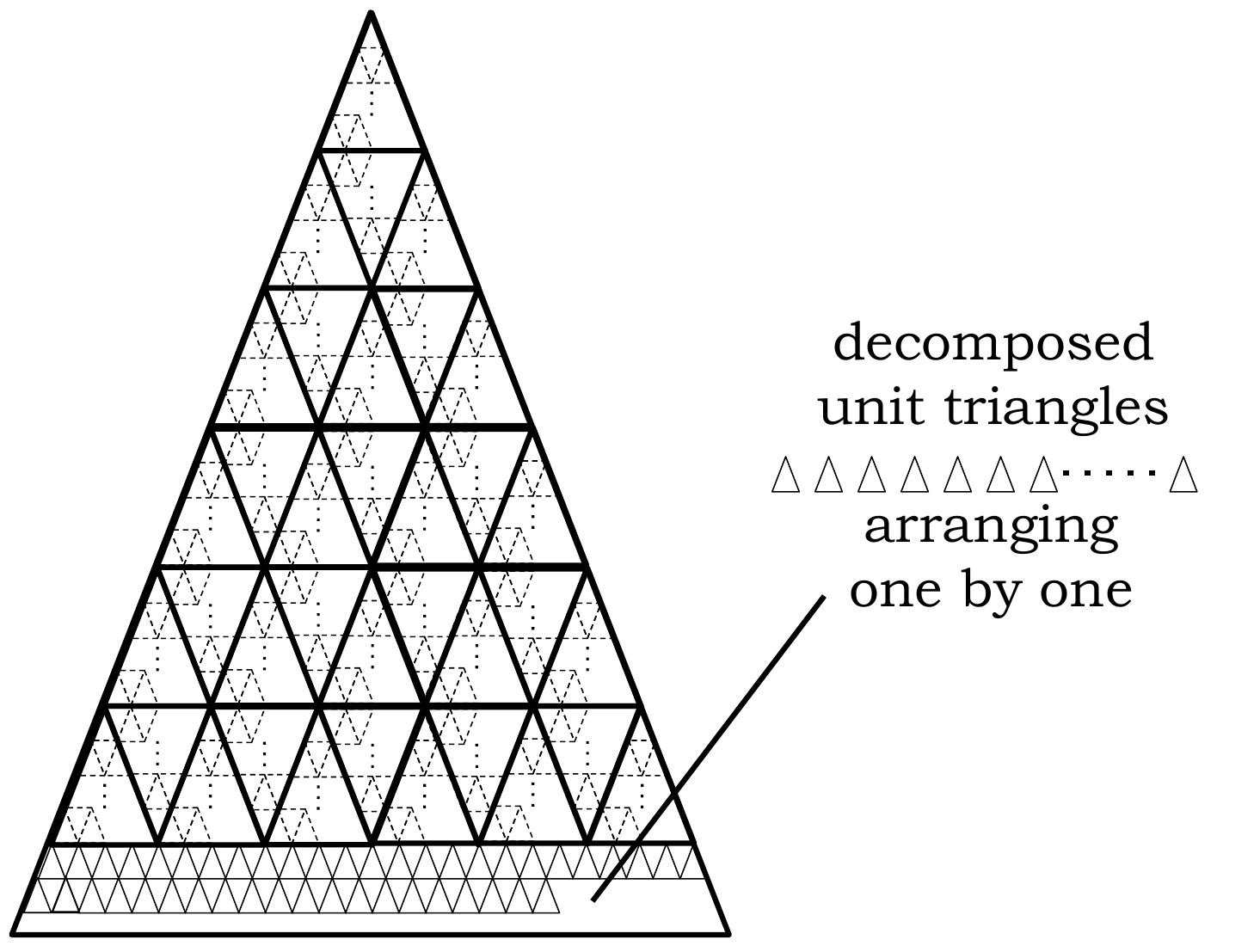}\\
      \figure The remaining carpets are decomposed into unit triangles, which are arranged into the uncarpeted space. \label{fig:57.99}
      \end{center}

By assumption, we know that there are $a^2$ unit triangles in this collection, and that $a^2$ unit triangles with base length 1 can perfectly fit into this room, in $a$ rows, where $a$ is an integer. Thus, there must be some rational number $\lambda$ such that $0<\lambda b<b$, or $0<\lambda<1$, and $\lambda b$ is a positive integer and where $\lambda b$ is the number of rows of unit triangles in the bottom section, as shown in Figure~\ref{fig:57.95}.

 \begin{center}
\includegraphics[width=.55\textwidth]{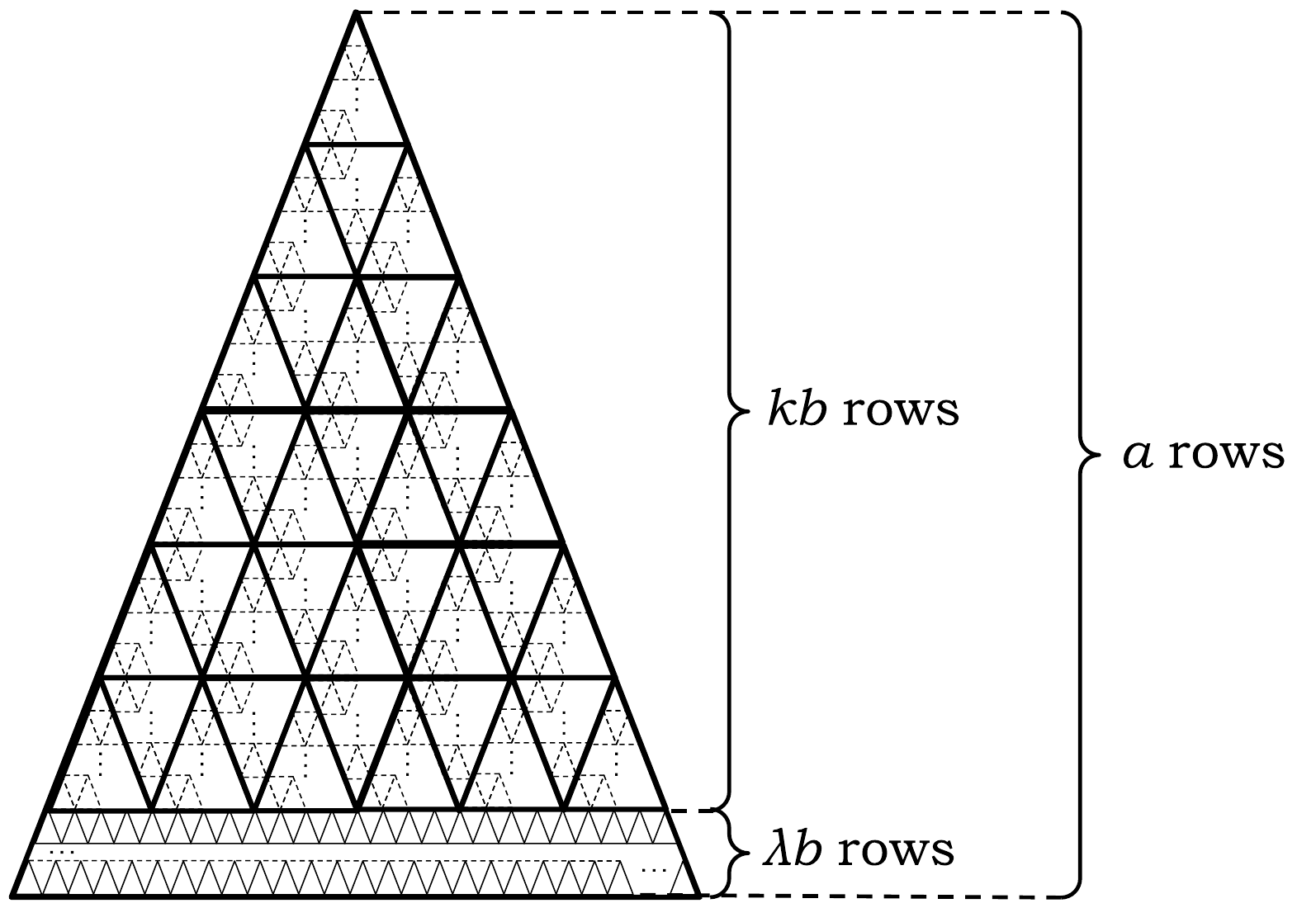}\\
      \figure $\lambda b$ rows of unit triangles would complete the tiled room. \label{fig:57.95}
      \end{center}

      Because $a=kb+\lambda b$, or $a=(k+\lambda)b$, it must also be true that $a^2=(k+\lambda)^2b^2$, or equivalently $a^2/b^2=(k+\lambda)^2$. However, by Theorem~\ref{thm:quadratic}, $a^2/b^2=n$. Therefore, $n$ must be equal to $(k+\lambda)^2$. However, $n$ is an integer by the definition of grationality, while $(k+\lambda)^2$ cannot be an integer. We know $(k+\lambda)^2$ to be a non-integer because $\lambda$, and thus $k+\lambda$, are rational non-integers. If $k+\lambda =x/y$ for positive integers $x$ and $y$, then $y\ne 1$, implying $y^2\ne 1$, making $x^2/y^2$ a rational non-integer. Therefore, it must be that $\lambda=0$, leaving no remaining floor space.

Now we are certain that with this arrangement, $n$ carpets will cover the triangular room's floor with no requirement of cutting and rearranging, given that each carpet is triangular with central angle $2\pi/n$ and base length $b$ and that the room is also such a triangle with a base length of $a$. In other words, $kb=a$ for some integer $k>1$. The number of rows of carpets will be $k$, as shown in Figure~\ref{fig:57.72}. 
 \begin{center}
\includegraphics[width=.4\textwidth]{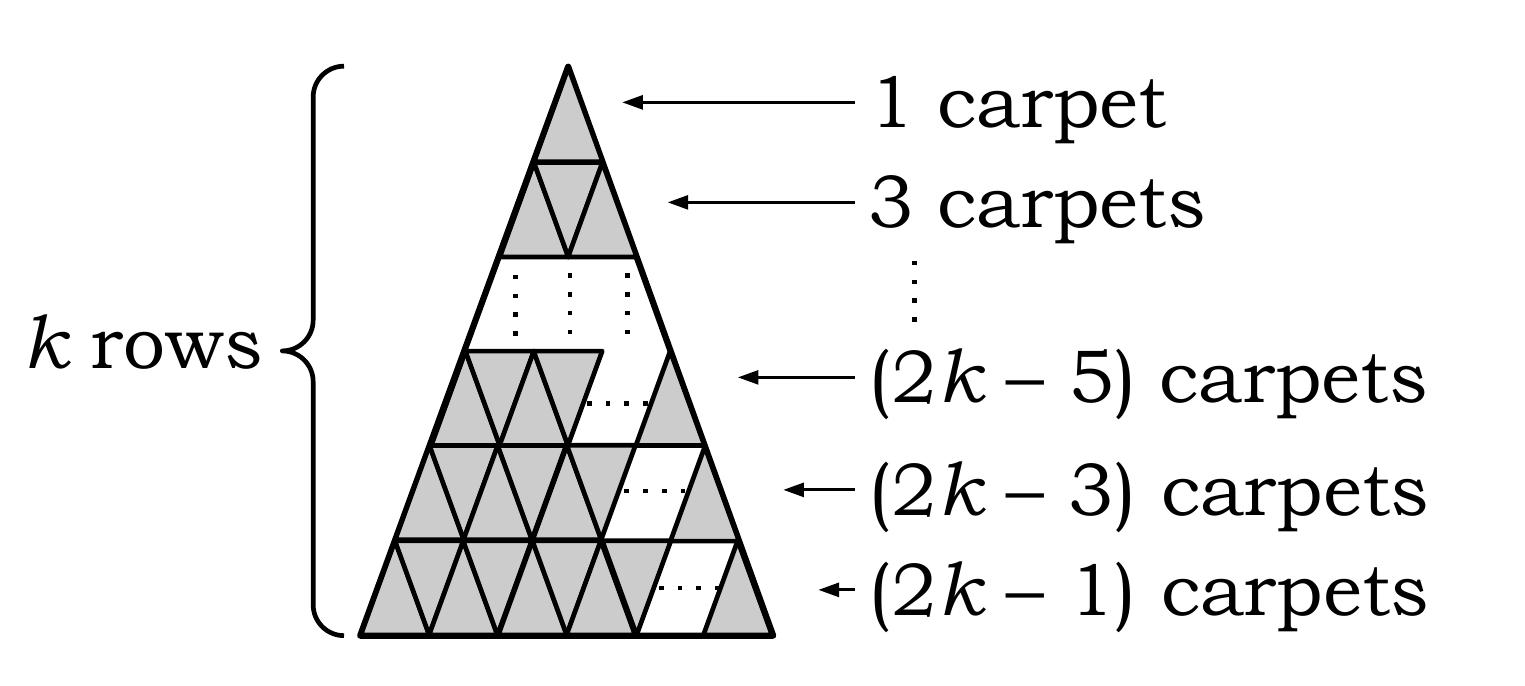}\\
      \figure The room can be tiled by $k$ rows of carpets.     \label{fig:57.72}
      \end{center}
Also, Figure~\ref{fig:57.72} shows that the number of carpets in each row is a sequence of odd numbers, from 1 to $(2k-1)$. Although well-known summation formulas can be used, I proceed with a spoon. Double this tiled room and reflect one of them so that the two triangular rooms form a parallelogram, as shown in Figure~\ref{fig:57.79}.

     \begin{center}
\includegraphics[width=.55\textwidth]{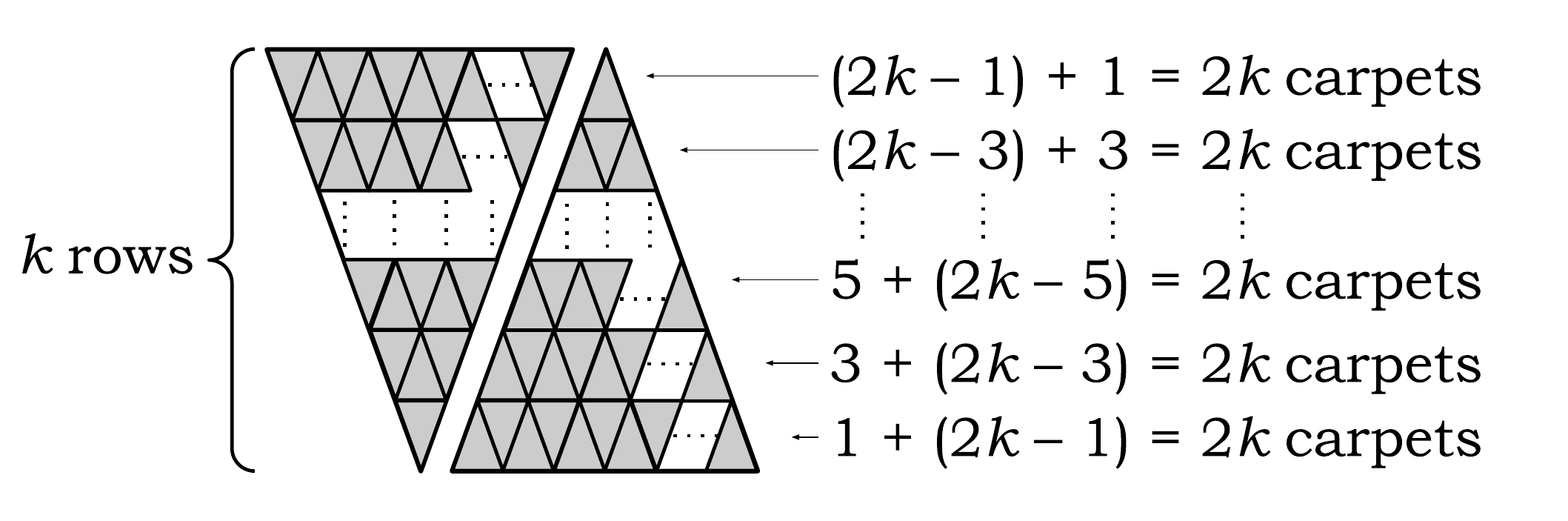}\\
      \figure Doubling the tiled room creates a parallelogram with $2k$ carpets in each row.        \label{fig:57.79}
      \end{center}

Because the parallelogram has $k$ rows of $2k$ carpets, the total number of carpet is $2k^2$, meaning the number of triangular carpets inside each of the two rooms is $k^2$. Thus, $n$ is a perfect square.
\end{proof}

\section{Conclusion}\label{conclusion}
The utensils for this article have been delimited to an assortment of spoons, such as proportions, tiling, carpets in rooms, and contradictions. I chose these restrictions, but why? 

Let me refer back to the analogous childhood story. What advantages were there to digging a swimming pool with a spoon for 3-year-old me?\begin{itemize}
    \item It allowed me to notice other cool stuff while digging in the dirt.
    \item I felt connected because I could hold my spoon so close.
    \item I felt autonomy. Since my tools were small, I had to do more work.
    \item It made a big task feel smaller.
    \item I had confidence. My father said I could do it.
    \item I felt safe. I knew he was there with a shovel if I needed him.
\end{itemize}

Turning toward mathematics, I believe there are similar advantages. For example, Plato's writings alluded to extremely simple mathematical tools; therefore, choosing to use simple tools, which rely more on logic than on machinery, allow us to connect with math history and with human intuition. Also, I believe this type of ``spoon math" to be welcoming for broader participation in mathematics because logic comes from within. If young learners can feel math resonate within their own logic and experiences, then I believe they will be less likely to give up and more likely to feel ownership and reward. 

As mathematics does not occur in a vacuum, but rather in the fabric of human nuance, I want to provide some personal details that contributed to the birth of this project. After my father's death in 2024, a new sensation of loss failed to fit into any of my intellectual boxes. Thus, I was impelled to escape into mathematics, a more comfortable world where logic is supposed to work. The funeral director had not yet seen a customer bring math proofs to a family consultation, but added, ``People grieve in their own ways," helping me realize the truth that mathematics was holding my hand. In his memory, I would like to draw one more connection. When I was 3 years old, I had begun to dig a swimming pool with a spoon. What gives a child that kind of confidence to take on a swimming pool project? It was my father's encouragement and the realization that he would back me up with heavier machinery when I needed help. We need that in math too. Be that person who propagates confidence in the abilities of others to contribute to the world of mathematics. 

\section*{Acknowledgements}
I want to acknowledge appreciation for the Mathematical Association of America. My term as Section Lecturer for the Southeastern Section motivated me to attack mathematics in the most accessible way possible, toward including more students in math engagement. As the concept of grationality was widely disseminated in talks in 2025, such as in~\cite{MAAtalk} and in multiple colloquium talks, I acknowledge (and hope) that others have already begun working on the further development of ideas around grationality. Moreover, David E. Dobbs sent me proofs of the grationality-iff-perfect-square conjecture, in which the techniques used were different from those this paper, with minimal overlap, if any. Although I viewed his proofs, the proofs in this paper were independent works. I wish him tremendous satisfaction and success in the dissemination of his work on the concept of grationality. To him and others who have also decided to take up the concept or the challenge to try spoon-math, I want to offer gratitude for engaging in this mathematical conversation. 

\bibliographystyle{abbrv}
\bibliography{references}
\bigskip

\bigskip

\end{document}